\documentclass[reqno, 12pt]{amsart}
\usepackage[T1]{fontenc}    
\usepackage[utf8]{inputenc} 
\usepackage{lmodern}       
\usepackage{mathrsfs}
\usepackage{amscd}
\usepackage{amsmath}
\usepackage{latexsym}
\usepackage{amsfonts}
\usepackage{amssymb}
\usepackage{amsthm}
\usepackage{graphicx}
\usepackage{makecell}
\usepackage{array}
\usepackage{booktabs}
\usepackage{multirow}
\usepackage{color,xcolor}
\usepackage{comment}
\allowdisplaybreaks[4] 

\usepackage[a4paper,margin=1in]{geometry}
\usepackage{fancyhdr}

\newtheorem{theorem}{Theorem}[section]

\newtheorem{corollary}[theorem]{Corollary}
\newtheorem{lemma}[theorem]{Lemma}

\newtheorem{remark}[theorem]{Remark}

\usepackage[colorlinks=true, linkcolor=blue]{hyperref} 
\numberwithin{equation}{section}

\author[Qiu]{Dong Qiu}
\address{School of Mathematical Sciences, Zhejiang University}
\email{qiudong@zju.edu.cn}

\author[Xu]{Xiang Xu}
\address{School of Mathematical Sciences, and Center for Interdisciplinary Applied Mathematics, Zhejiang University}
\email{xxu@zju.edu.cn}

\author[Ye]{Yeqiong Ye}
\address{School of Mathematical Sciences, Zhejiang University}
\email{yeyeqiong@zju.edu.cn}

\author[Zhou]{Ting Zhou}
\address{School of Mathematical Sciences, Zhejiang University}
\email{ting\_zhou@zju.edu.cn}

\title[Gauge symmetry and uniqueness for the JMGT equation]{Gauge symmetry and uniqueness in inverse problems for the JMGT equation}

\begin{document}

\begin{abstract}
In this paper, we study an inverse boundary value problem for the Jordan--Moore--Gibson--Thompson equation on a simple Riemannian manifold. We consider an all boundary measurement map that maps Dirichlet boundary data and initial data to the corresponding Neumann-type boundary data and final-time data. Our main result shows that the nonlinear acoustic coefficient $\beta$ is uniquely determined by this measurement map, and the linear damping coefficients $\alpha$ and $q$, along with the internal source term $F$, can be recovered up to a gauge symmetry. As a corollary, we also identify several cases in which all coefficients are uniquely recovered. The proof relies on  the method of first order and second order linearization and on the construction of geometric optics solutions. In the intermediate step,  we establish the unique recovery of the lower-order coefficients in the linearized MGT equation.
\end{abstract}
\maketitle

\section{Introduction}

\subsection{Statement of the problem}
In this article, we consider an inverse boundary value problem for the Jordan--Moore--Gibson--Thompson (JMGT) equation, a nonlinear hyperbolic model arising in ultrasound imaging. 
By incorporating thermoviscous dissipation and relaxation effects, the JMGT equation captures finite-speed wave propagation together with frequency-dependent attenuation (see  \cite{jordan2008nonlinear,jordan2014second,moore1960propagation,Thompson1972}), and serves as an important nonlinear acoustic model with applications in medical ultrasound imaging. 

Let $(M,g)$ be a compact connected Riemannian manifold of dimension $n\ge2$ with a smooth boundary $\partial M$. We set $\mathcal{Q}:=(0,T) \times M$ and $\Gamma=(0,T) \times \partial M$.
Consider the generalized JMGT equation with sources in the following form
\begin{equation}\label{JMGTso}
\left\{
\begin{aligned}
&\partial_t^3u + \alpha \partial_{t}^2 u  - \Delta_g \partial_t u-c^2\Delta_g u+q \partial_t u = \partial_t \left(\beta(t,x)(\partial_tu)^2  \right) +F
    && \text{in } \mathcal{Q}, \\
&u= f
    && \text{on } \Gamma, \\
&u(0,\cdot)=h^{(0)},  \ \partial_t u(0,\cdot) = h^{(1)},\ \partial_t^2u(0,\cdot)=h^{(2)} 
    && \text{in } M.
\end{aligned}
\right.
\end{equation}
Here $\alpha$ and $q$ describe damping effects in the medium, $c$ is the sound speed, $g$ encodes the anisotropic geometry of the medium, $\beta$ characterizes the acoustic nonlinearity of the medium, and $F$ denotes an applied source term. In particular, $\alpha,c$ and $q$ govern the linear propagation and attenuation, while $\beta$ is responsible for nonlinear effects in ultrasound propagation.

For a fixed sound speed $c$, we define the corresponding \textit{all boundary measurement map} as
\begin{equation}
\mathcal L_{\alpha,q,\beta,F}: \big(f,h^{(0)},h^{(1)},h^{(2)}\big) \to \bigl((\partial_{\nu}\partial_t u +c^2\partial_{\nu }u)|_{\Gamma},u(T,\cdot),\partial_t u(T,\cdot),\partial_t^2 u(T,\cdot) \bigr),
\end{equation}
where $u$ is the solution of \eqref{JMGTso} and $\nu$ denotes the unit outward normal vector of $\partial M$ with respect to the Riemannian metric $g$. 
We refer the reader to Section~\ref{well} for local well-posedness of the JMGT equation \eqref{JMGTso}, which guarantees $\mathcal L_{\alpha,q,\beta,F}$ is well-defined.  
The inverse problem that we consider is the following:

\textbf{Inverse Problem}: Does the all boundary measurement map $\mathcal L_{\alpha,q,\beta,F}$  determine $\alpha,q,\beta,F$ uniquely or  up to a natural gauge symmetry?

Before presenting our main results, we introduce the following notation and definitions. 
For a nonnegative integer $m$, we define the following function spaces:
\begin{align*}
&A_{m} = \bigcap_{k=0}^{m+1} C^{k}\bigl([0, T]; H^{m+\frac{3}{2}-k}(\partial M)\bigr) \cap H^{m+2}(\Gamma), \
N_{m+1} =
\bigcap_{k=1}^{m+1}H^{k}\bigl([0, T]; H^{m+1-k}(\partial M)\bigr), \\
& O=A_m \times H^{m+2}(M)\times H^{m+1}(M)\times H^{m}(M),\\
& R=\{(f,h^{(0)},h^{(1)},h^{(2)})\in O: (f,h^{(0)},h^{(1)},h^{(2)})   \text{ satisfies the compatibility condition } \eqref{compaMGT}   \}.
\end{align*}
We work in the \textit{energy spaces} $E^m$, defined by 
\begin{equation*}
E^m=\bigcap\limits_{k=0}^{m} C^k\left([0,T];H^{m-k}(M)\right),
\end{equation*}
equipped with the norm 
\begin{equation*}
\|u\|_{E^m}:=\sup_{t\in[0,T]}\sum_{k=0}^{m}\|\partial_t^k u(\cdot,t)\|_{H^{m-k}(M)}.
\end{equation*}
The space $E^m$ is an algebra if $m >n+1$ (see \cite{choquet-bruhat_general_2009}). Moreover, it satisfies the norm estimate
\begin{equation*}
||uv||_{E^m}\leq C_m ||u||_{E^m}||v||_{E^m},  \quad \text{for all } u,v\in E^m,
\end{equation*}
where $C_m$ is a constant depending on $m$. 

For the compatibility condition, we define the admissible function spaces for potentials $\alpha$ and $q$:
\begin{align*}
E^m_{0,k}=\{ a\in E^m:  \partial_t^j a=0 \text{ on } \{t=0\}\times  \partial M \text{ for } j=0,1,\cdots,k  \}.
\end{align*}

\subsection{Main results} Throughout this paper, we assume that $m$ is an integer, $m>n+1$, and $0<c_0\le c(x)\le c_1 \ \text{in }  M$. 
\begin{theorem}\label{prop1.1}
Assume that $c\in C^\infty(M)$, $\alpha, q,F \in E_{0,m-2}^m$ and $\beta \in E_{0,m-1}^{m+1}$.
Let $(f_0,h^{(0)}_0,h^{(1)}_0,h^{(2)}_0)\in R$ and let $u_0\in E^{m+2}$ be the unique solution to  \eqref{JMGTso} with boundary and initial data
$(f_0,h^{(0)}_0,h^{(1)}_0,h^{(2)}_0)$.
Then there exists $\varepsilon>0$ such that, for any
$(f,h^{(0)},h^{(1)},h^{(2)})\in G_\varepsilon$,
where
\begin{align}
G_\varepsilon=\big\{(f,h^{(0)},h^{(1)},h^{(2)})\in R: \ &\|f-f_0\|_{H^{m+2}(\Gamma)}
+\|h^{(0)}-h^{(0)}_0\|_{H^{m+2}(M)} \nonumber
\\
&
+\|h^{(1)}-h^{(1)}_0\|_{H^{m+1}(M)}
+\|h^{(2)}-h^{(2)}_0\|_{H^m(M)}
<\varepsilon
\big\},   \label{Gep}
\end{align}
there exists a unique solution $u\in E^{m+2}$ to \eqref{JMGTso} with
$\partial_\nu u\in N_{m+1}$,
satisfying
\begin{equation}\label{wellestimate}
\|u-u_0\|_{E^{m+2}}+\|\partial_\nu (u-u_0)\|_{N_{m+1}}<C\varepsilon.
\end{equation}
\end{theorem}

For the inverse problems below, we further assume that $(M,g)$ is simple, that is $(M,g)$ is non-trapping, the boundary $\partial M$ is strictly convex and there are no conjugate points.
\begin{theorem}\label{thme2}
Assume that $c\in C^\infty(M)$ is fixed, $\alpha_j, q_j, F_j\in E_{0,m-2}^m$ and $\beta_j \in E_{0,m-1}^{m+1}$, for $j=1,2$.
Let $(f_0,h^{(0)}_0,h^{(1)}_0,h^{(2)}_0)\in R$ and let $u_{0,j} \in E^{m+2}$ be the unique solution to \eqref{JMGTso} with boundary and initial data $(f_0,h^{(0)}_0,h^{(1)}_0,h^{(2)}_0)$, with $(\alpha,q,\beta,F)$ replaced by $(\alpha_j,q_j,\beta_j,F_j)$, for $j=1,2$.
Let $\mathcal L_{\alpha_j,q_j,\beta_j,F_j}$ denote the corresponding all boundary measurement map of \eqref{JMGTso} with
$(\alpha,q,\beta,F)$
replaced by
$(\alpha_j,q_j,\beta_j,F_j)$, for $j=1,2$.
Choose $\varepsilon>0$ sufficiently small so that both maps 
$\mathcal L_{\alpha_1,q_1,\beta_1,F_1}$ and 
$\mathcal L_{\alpha_2,q_2,\beta_2,F_2}$ are well-defined on 
$G_\varepsilon$.
If there exists a nonempty neighborhood $\mathcal N \subset G_\varepsilon$ containing $(f_0,h^{(0)}_0,h^{(1)}_0,h^{(2)}_0)$ such that
\[
\mathcal L_{\alpha_1,q_1,\beta_1,F_1}\big(f,h^{(0)},h^{(1)},h^{(2)}\big)
=
\mathcal L_{\alpha_2,q_2,\beta_2,F_2}\big(f,h^{(0)},h^{(1)},h^{(2)}\big),
\qquad \forall \big(f,h^{(0)},h^{(1)},h^{(2)}\big)\in \mathcal N,
\]
then we have
\[
\beta_2=\beta_1=:\beta,\qquad
\alpha_2=\alpha_1+2\beta\,\partial_t\psi,\qquad
q_2=q_1+2\partial_t(\beta\partial_t \psi),
\]
and
\[
F_2
=
F_1+\partial_t^3\psi+\alpha_1\partial_t^2\psi-\Delta_g\partial_t\psi-c^2\Delta_g\psi
+q_1\partial_t\psi+\partial_t\bigl(\beta(\partial_t\psi)^2\bigr),
\]
for some $\psi\in E^{m+2}$ with $\psi|_{\Gamma}=(\partial_{\nu} \partial_t \psi+c^2 \partial_{\nu} \psi)|_{\Gamma}=0$, $\psi(0,\cdot)=\partial_t \psi(0,\cdot)=\partial_t^2 \psi(0,\cdot)=0$ in $M$ and $\psi(T,\cdot)=\partial_t \psi(T,\cdot)=\partial_t^2\psi(T,\cdot)=0$ in $M$.
\end{theorem}

\begin{corollary}\label{cor1.3}
Under the same assumptions as in Theorem \ref{thme2}, if one of the following conditions is satisfied:
\begin{enumerate}
\item[(i)] $\alpha_1=\alpha_2$, \quad $\beta\neq 0$ in $\mathcal Q$,
\item[(ii)] $q_1=q_2$, \quad $\beta\neq 0$ in $\mathcal Q$,
\item[(iii)] $F_1=F_2$ in $\mathcal Q$,
\end{enumerate}
then we have
\[
\beta_1=\beta_2,\qquad q_1=q_2,\qquad \alpha_1=\alpha_2,\qquad F_1=F_2
\quad \text{in } \mathcal Q.
\]
\end{corollary}

\subsection{Related works and previous literature}

Inverse problems for nonlinear hyperbolic equations have been studied extensively since the pioneering work \cite{kurylev_inverse_2018}.  The paper \cite{kurylev_inverse_2018} observed that nonlinearity can be used as a powerful tool in inverse problems for nonlinear wave equations and showed that local measurements for the scalar wave equation with a quadratic nonlinearity determine the conformal class of a globally hyperbolic four-dimensional Lorentzian manifold. Their approach, now known as the higher order linearization method, has been used widely in \cite{feizmohammadi2020inverse,feizmohammadi2022recovery,lassas2021inverse,lassas_uniqueness_2022,lassas_stability_2025,uhlmann_inverse_2021}.  We refer the reader to \cite{hintz2022dirichlet,lassas2018inverse,uhlmann2022inverse,wang2019inverse} for further results. We  also refer the reader to \cite{oksanen_inverse_2024} for a unified approach treating general real principal type differential operators.
See \cite{harju2025x,Lassas2024,maddah2025numerical,uhlmann2020convolutional} for the computational aspects of the higher order linearization method.

For more physically relevant models, inverse problems arising in nonlinear ultrasound imaging have attracted considerable attention, including those for the Westervelt equation and the JMGT equation. The paper \cite{acosta_nonlinear_2022} proved that the Dirichlet-to-Neumann (DtN) map determines the nonlinearity in the Westervelt equation without damping terms.  For the related studies on inverse problems for the Westervelt equation, we refer the reader to \cite{acosta_simultaneous_2025,eptaminitakis2024weakly,li2024inverse,uhlmann2023}. Among these,  the two papers most closely  related to ours are as follows. The paper \cite{fu2023inverse} considered the unique recovery of a Westervelt-type nonlinearity from the all boundary measurement map in the JMGT equation. Our previous paper \cite{qiu2026inverse} addressed the unique recovery of Westervelt-type and Kuznetsov-type nonlinearity from the DtN map for the JMGT equation in the Riemannian manifold setting. Recent works on inverse problems for the JMGT and MGT equations include \cite{fu_partial_2026,fu_stability_2025,kaltenbacher_acoustic_2025,kaltenbacher_imaging_2025}.

Inverse source problems for nonlinear equations have recently become an active research direction.
The papers \cite{kian_determining_2024,liimatainen_uniqueness_2024} show that inverse source problems for nonlinear equations may exhibit gauge invariance and nonlinearity can help break this gauge.
Inverse source problems have also been studied in broader nonlinear settings in \cite{lassas2025gaussian,liimatainen2026inverse,liimatainen2025inverse2,liimatainen2025inverse,D2026}, including the semilinear wave equation, the fully nonlinear Monge--Amp\`ere equation, the quasilinear prescribed mean curvature equation and the generalized Calder\'on problem.

Compared with the existing results in \cite{fu2023inverse,qiu2026inverse}, 
which focus on the recovery of nonlinear coefficients, we study a simultaneous 
coefficient-and-source recovery problem for the JMGT equation.
In the simple Riemannian manifold setting, we show that the nonlinear coefficient $\beta$ is uniquely determined by the all boundary measurement map, whereas the lower-order coefficients $\alpha$ and $q$, together with the source term $F$, can only be recovered up to a natural gauge symmetry. Our analysis relies on a combination of  first order linearization and second order linearization, and the construction of geometric optics solutions to the linearized MGT equation. As a corollary, under additional assumptions, we further obtain the unique recovery of all unknown coefficients.

An intermediate step of independent interest is the unique recovery of the lower-order coefficients in the linearized MGT equation. 
Compared with existing results on inverse boundary value problems for linear MGT equations in Euclidean domains \cite{fu_partial_2026,fu_stability_2025}, we extend the uniqueness result to the setting of simple Riemannian manifolds.

The paper is organized as follows. In Section \ref{well}, we establish the local well-posedness of the boundary value problem \eqref{JMGTso} near the given solution. In Section \ref{go}, we present the construction of geometric optics solutions for the linearized MGT equation. In Section \ref{uniq}, we prove the unique recovery of two lower-order coefficients in the linearized MGT equation. In Section \ref{proo}, we  provide the proof of Theorem \ref{thme2} and the proof of Corollary \ref{cor1.3}. 

\section{Local Well-posedness of the JMGT equation}\label{well}

In this section, we establish the local well-posedness of the JMGT equation in a neighborhood of a given solution.
We begin with the following linear MGT equation:
\begin{equation}\label{MGTl}
\left\{
\begin{aligned}
&\partial_t^3v + \alpha(t,x) \partial_{t}^2 v  - \Delta_g \partial_t v-c^2(x)\Delta_g v+q(t,x) \partial_t v = F(t,x)
    && \text{in } \mathcal{Q}, \\
&v= f(t,x)
    && \text{on } \Gamma, \\
&v(0,x)=h^{(0)}(x), \partial_t v(0,x) = h^{(1)}(x), \partial_t^2 v(0,x)=h^{(2)}(x) 
    && \text{in }  M,
\end{aligned}
\right.
\end{equation}
where $\alpha,q, F \in E_{0,m-2}^{m}$ and $c\in C^{\infty}(M)$. Since $\alpha,q, F \in E_{0,m-2}^m$, all terms involving the boundary traces of $\partial_t^j\alpha$, $\partial_t^j q$ and $\partial_t^j F$, $0\le j\le m-2$, vanish in the compatibility conditions below.
We  introduce the compatibility conditions associated with \eqref{MGTl}. On the boundary $\partial M$, the quadruple $(f,h^{(0)},h^{(1)},h^{(2)})$ is said to satisfy  the MGT compatibility condition up to order $m+1$  if $(f,h^{(0)},h^{(1)},h^{(2)})\in O$ satisfies 
\begin{align} \label{compaMGT}
\left\{
\begin{aligned}
&f(0,x) = h^{(0)}(x),\\
&\partial_t f(0,x) = h^{(1)}(x),\\
&\partial_t^2 f(0,x) = h^{(2)}(x),\\
&\partial_t^3 f(0,x)  - \Delta_g h^{(1)}(x)-c^2(x) \Delta_g h^{(0)}(x) = 0,\\
&\partial_t^4f(0,x)-\Delta_g h^{(2)}(x)-c^2(x) \Delta_g h^{(1)}(x)=0,   \\
&\vdots\\
&\text{up to the temporal order } m+1 .
\end{aligned}
\right.
\end{align}
We also use \eqref{compaMGT} as the compatibility condition for the JMGT equation \eqref{JMGTso}. Indeed, since $\beta\in E_{0,m-1}^{m+1}$, the nonlinear term
$\partial_t(\beta(\partial_t u)^2)$ does not contribute to the compatibility conditions up to the required order on $\{t=0\}\times\partial M$.

We recall a well-posedness result for the  MGT equation without the damping term $q$.
\begin{lemma} \label{lemma2.1}
Suppose that $\alpha, c\in C^\infty(M), F\in E^m$, $(f,h^{(0)},h^{(1)},h^{(2)}) \in O$   
and the quadruple $(f,h^{(0)},h^{(1)},h^{(2)})$ satisfies the MGT compatibility condition:
\begin{align} \label{comp2.1}
\left\{
\begin{aligned}
&f(0,x) = h^{(0)}(x),\ \partial_t f(0,x) = h^{(1)}(x), \ \partial_t^2 f(0,x) = h^{(2)}(x)  &\text{on }\partial M\\
&\partial_t^3 f(0,x) + \alpha(x) h^{(2)}(x) - \Delta_g h^{(1)}(x) - c^2(x) \Delta_g h^{(0)}(x)= F(0,x) &\text{on }\partial M,\\
&\partial_t^4f(0,x)+\alpha(x) \partial_t^3f(0,x)-\Delta_g h^{(2)}(x)-c^2(x)\Delta_g h^{(1)}(x)=\partial_tF(0,x)&\text{on }\partial M,   \\
&\vdots \\
&\text{up to the temporal order } m+1.
\end{aligned}
\right.
\end{align}
Then the linear system 
\begin{equation}\label{MGT2.1}
\left\{
\begin{aligned}
&\partial_t^3v + \alpha(x) \partial_{t}^2 v  - \Delta_g \partial_t v-c^2(x)\Delta_g v = F(t,x)
    && \text{in } \mathcal{Q}, \\
&v= f(t,x)
    && \text{on } \Gamma, \\
&v(0,x)=h^{(0)}(x), \partial_t v(0,x) = h^{(1)}(x), \partial_t^2 v(0,x)=h^{(2)}(x) 
    && \text{in }  M,
\end{aligned}
\right.
\end{equation}
admits a unique solution $v \in E^{m+2}$ with $\partial_{\nu} v \in N_{m+1}$ such that 
\begin{align}\label{esm2.1}
 \left\|v \right\|_{E^{m+2}}+\left\| \partial_{\nu}v  \right\|_{N_{m+1}}\leq & Ce^{CT} \Big( \left\|h^{(0)}\right\|_{H^{m+2}(M)}+ \left\|h^{(1)}\right\|_{H^{m+1}(M)}+\left\|h^{(2)}\right\|_{H^{m}(M)} \nonumber\\
&+\sum_{k=0}^m \left\| \partial_t^k F \right\|_{L^2([0,T];H^{m-k}(M))}+\left\|f\right\|_{H^{m+2}(\Gamma)}   \Big).
\end{align}
\end{lemma}
\begin{proof}
By \cite[Lemma 2.2]{qiu2026inverse}, for data $(f,h^{(0)},h^{(1)},h^{(2)})$ satisfying the compatibility conditions \eqref{comp2.1}, equation \eqref{MGT2.1} admits a unique solution $v \in E^{m+2}$ with $\partial_{\nu} v \in N_{m+1}$.
Moreover, for the case $m=0$, we have
\begin{align}
&\int_M \bigl(v_{tt}^2 + |\nabla_g v_t|_g^2 + v_t^2 + |\nabla_g v|_g^2 + |\Delta_g v|^2\bigr)\mathrm{d}V_g \nonumber \\
& \leq C\Bigl( \|h^{(0)}\|_{H^2(M)}^2 + \|h^{(1)}\|_{H^1(M)}^2 + \|h^{(2)}\|_{L^2(M)}^2 \Bigr)  \nonumber \\
&+ C \int_0^t \int_M \bigl(v_{tt}^2 +|\nabla_g v_t|_g^2 +  v_t^2 + |\nabla_g v|_g^2 + |\Delta_g v|^2\bigr)\mathrm{d}V_g\mathrm{d}l \nonumber \\
& + C\Bigl( \|F\|_{L^2(\mathcal{Q})}^2 + \|f\|_{H^2([0,T];L^2(\partial M))}^2 + \|f\|_{H^1([0,T];H^1(\partial M))}^2 \Bigr).\label{esm2}
\end{align}
Applying  Gronwall's inequality and combining \eqref{esm2} and the boundary regularity, we obtain
\begin{equation} \label{energy}
\begin{aligned}
\|v\|_{E^2}^2 + \|\partial_{\nu}v\|_{H^1([0,T];L^2(\partial M))}^2
&\le Ce^{CT}\Bigl(
\|h^{(0)}\|_{H^2(M)}^2 + \|h^{(1)}\|_{H^1(M)}^2 + \|h^{(2)}\|_{L^2(M)}^2 \\
&\qquad\quad
+ \|F\|_{L^2(\mathcal{Q})}^2 + \|f\|_{H^2(\Gamma)}^2
\Bigr).
\end{aligned}
\end{equation}
To obtain the higher-order estimate, we adopt the bootstrap argument. For any positive integer $1\le k\le m$, set $\hat{v}= \partial_t^k v$, which satisfies the following equation:
\begin{equation*}
\left\{
\begin{aligned}
&\partial_t^3\hat{v} + \alpha(x) \partial_{t}^2 \hat{v}  - \Delta_g \partial_t \hat{v}-c^2(x)\Delta_g \hat{v} = \partial_t^kF
    && \text{in } \mathcal{Q}, \\
&\hat{v}= \partial_t^kf
    && \text{on } \Gamma. \\
\end{aligned}
\right.
\end{equation*}
Since $F \in E^m, (f,h^{(0)},h^{(1)},h^{(2)}) \in O$, we have that
\begin{align*}
    &\partial_t^kF \in L^2([0,T];H^{m-k}(M)),\partial_t^kf \in A_{m-k},\\
    &\left(\hat{v}(0,\cdot),\partial_t\hat{v}(0,\cdot),\partial_t^2\hat{v}(0,\cdot)\right) \in H^{m+2-k}(M) \times H^{m+1-k}(M) \times H^{m-k}(M).
\end{align*}
The compatibility conditions of order $m+1-k$ also hold. Combining the estimates for all 
$k$, we can prove that estimate \eqref{esm2.1} holds.
\end{proof}

\begin{lemma} \label{lemma2.2}
Suppose that $c\in C^\infty(M)$, $\alpha, q,F\in E_{0,m-2}^{m}$, $(f,h^{(0)},h^{(1)},h^{(2)}) \in R$. Then the linear system \eqref{MGTl} admits a unique solution $v \in E^{m+2}$ with $\partial_{\nu} v \in N_{m+1}$ such that 
\begin{align*}
& \left\|v \right\|_{E^{m+2}}+\left\| \partial_{\nu}v  \right\|_{N_{m+1}}\leq \\ &C  
\left( \left\|h^{(0)}\right\|_{H^{m+2}(M)}+ \left\|h^{(1)}\right\|_{H^{m+1}(M)}+\left\|h^{(2)}\right\|_{H^{m}(M)}+\left\|F\right\|_{E^m}+\left\|f\right\|_{H^{m+2}(\Gamma)}   \right).
\end{align*}
\end{lemma}
\begin{proof}
For any $w \in  E^{m+2}$, consider the following equation:
\begin{equation}\label{MGT2.2}
\left\{
\begin{aligned}
&\partial_t^3v - \Delta_g \partial_t v-c^2(x)\Delta_g v = F(t,x)-\alpha(t,x) \partial_{t}^2 w-q(t,x) \partial_t w
    && \text{in } \mathcal{Q}, \\
&v= f(t,x)
    && \text{on } \Gamma, \\
&v(0,x)=h^{(0)}(x), \partial_t v(0,x) = h^{(1)}(x), \partial_t^2 v(0,x)=h^{(2)}(x) 
    && \text{in }  M.
\end{aligned}
\right.
\end{equation}
By the assumptions, one can show that $\tilde{F}=F-\alpha \partial_{t}^2 w-q \partial_t w \in E^m$ and $\partial_t^k\tilde{F}=\partial_t^kF$ on $\{t=0\} \times \partial M$ for $k=0,1,\cdots,m-2$. Then we can verify the compatibility conditions \eqref{compaMGT} hold.\\
By Lemma \ref{lemma2.1}, \eqref{MGT2.2} admits a unique solution $v\in E^{m+2}$ and $\partial_{\nu}v \in N_{m+1}$. Moreover,
\begin{align}\label{eq:v-es}
 \left\|v \right\|_{E^{m+2}}+\left\| \partial_{\nu}v  \right\|_{N_{m+1}}\leq & C \Big( \left\|h^{(0)}\right\|_{H^{m+2}(M)}+ \left\|h^{(1)}\right\|_{H^{m+1}(M)}+\left\|h^{(2)}\right\|_{H^{m}(M)}  \nonumber\\ & +\sum_{k=0}^m \left\| \partial_t^k \tilde{F} \right\|_{L^2([0,T];H^{m-k}(M))}+\left\|f\right\|_{H^{m+2}(\Gamma)}   \Big).
\end{align}
Define the mapping
\begin{equation*}
    \mathcal{L}:E^{m+2} \to E^{m+2}, \quad \mathcal{L}(w)=v,
\end{equation*}
where $v$ is the solution to \eqref{MGT2.2} associated with $w\in E^{m+2}$. For any $w_1,w_2 \in E^{m+2}$, denote by $v_1$ and $v_2$ the associated solutions to \eqref{MGT2.2}. By \eqref{esm2.1}, we obtain that
\begin{align*}
&\| v_1-v_2 \|_{E^{m+2}}+ \| \partial_{\nu}(v_1-v_2) \|_{N_{m+1}}  \\ 
&\leq C e^{CT}\sum_{k=0}^m \left\| \partial_t^k \tilde{F}_1- \partial_t^k \tilde{F}_2\right\|_{L^2([0,T];H^{m-k}(M))} \leq C \sqrt{T}e^{CT}\sum_{k=0}^m \left\| \partial_t^k \tilde{F}_1- \partial_t^k \tilde{F}_2\right\|_{C([0,T];H^{m-k}(M))}\\
& \leq C \sqrt{T}e^{CT}\left\|\alpha \partial_{t}^2 (w_1-w_2)+q \partial_t (w_1-w_2)\right\|_{E^m}\\
&\leq C \sqrt{T}e^{CT}\left(\|\alpha\|_{E^m}+\|q\|_{E^m}\right)\|w_1-w_2\|_{E^{m+2}}.
\end{align*}
If $T$ is sufficiently small so that $C \sqrt{T}e^{CT}\left(\|\alpha\|_{E^m}+\|q\|_{E^m}\right) <1$, then by the Banach fixed point theorem, $\mathcal{L}$ has a unique fixed point $v \in E^{m+2}$.\\
Using the same scaling argument as before, we obtain that
\begin{align*}
    \sum_{k=0}^m \left\| \partial_t^k \tilde{F} \right\|_{L^2([0,T];H^{m-k}(M))} \le  C\|F\|_{E^m}+C\sqrt{T}\left(\|\alpha\|_{E^m}+\|q\|_{E^m}\right)\|v\|_{E^{m+2}}.
\end{align*}
Choosing $T$ small enough such that $C \sqrt{T}\left(\|\alpha\|_{E^m}+\|q\|_{E^m}\right) < \frac{1}{2}$, we can absorb the $v$-dependent terms of \eqref{eq:v-es} from the right-hand side to the left-hand side to obtain the desired estimate.\\
Since \eqref{MGTl} is a linear equation, a standard continuation argument allows us to extend this local solution step-by-step to any $T>0$. 
\end{proof}

 We now turn to the proof of Theorem \ref{prop1.1}.
\begin{proof}
For $w \in B_{\rho}:=\left\{u \in E^{m+2}: \|u\|_{E^{m+2}} \leq \rho\right\}$, let $\tilde u$ be the solution of 
\begin{equation}\label{fixpoint}
\left\{
\begin{aligned}
&\partial_t^3 \tilde u + (\alpha-2\beta \partial_tu_0) \partial_{t}^2 \tilde u  - \Delta_g \partial_t \tilde u-c^2\Delta_g \tilde u +\left(q-2\partial_t(\beta\partial_tu_0)\right)\partial_t \tilde{u} = \partial_t\left(\beta(\partial_t w)^2\right)
    && \text{in } \mathcal{Q}, \\
&\tilde u= f-f_0
    && \text{on } \Gamma, \\
&\tilde u(0,x)=h^{(0)}-h^{(0)}_0, \ \partial_t \tilde u(0,x) = h^{(1)}-h^{(1)}_0, \  \partial_t^2 \tilde u(0,x)=h^{(2)}-h^{(2)}_0
    && \text{in }  M,
\end{aligned}
\right.
\end{equation}
where $\alpha-2\beta \partial_tu_0,q-2\partial_t(\beta\partial_tu_0) \in E_{0,m-2}^m$.
Since $E^m$ is an algebra and $\beta \in E^{m+1}$, one can show that 
\begin{align*}
&L(t,x):=\partial_t\left(\beta(\partial_t w)^2\right)  \in E^m, \\ 
\partial_t^k L=0, \ \
 \partial_t^k(\alpha-2\beta \partial_tu_0)=0, \
 &\text{ and } \partial_t^k(q-2\partial_t(\beta\partial_tu_0))=0 \ \ \ \text{ on } \{t=0\}\times\partial M,
\end{align*}
for $k=0,1,\cdots,m-2$. Then we can verify that  $(f-f_0,h^{(0)}-h^{(0)}_0,h^{(1)}-h^{(1)}_0,h^{(2)}-h^{(2)}_0)$ satisfies the MGT compatibility conditions \eqref{compaMGT} for the MGT equation \eqref{fixpoint}.
Moreover, from Lemma \ref{lemma2.2}, we have 
\begin{align}
&\| \tilde u \|_{E^{m+2}}+ \| \partial_{\nu}\tilde u \|_{N_{m+1}} 
\leq C ( \| \partial_t\left(\beta(\partial_t w)^2\right) \|_{E^m} \nonumber\\
&+\|f-f_0\|_{H^{m+2}(\Gamma)}+\|h^{(0)}-h^{(0)}_0\|_{H^{m+2}(M)}+\|h^{(1)}-h^{(1)}_0\|_{H^{m+1}(M)}+\|h^{(2)}-h^{(2)}_0\|_{H^m(M)})\nonumber  \\
&\leq C (\|\beta(\partial_t w)^2  \|_{E^{m+1}} +\varepsilon) \leq C (\| w\|_{E^{m+2}}^2+\varepsilon) \leq C (\rho^2+\varepsilon). \label{estimate2}
\end{align} 
For $\rho,\varepsilon$ sufficiently small, we have $\tilde u \in B_{\rho}$. Then we define a map
\begin{equation*}
\mathcal{G}: B_{\rho}\to B_{\rho}, \quad  \mathcal{G}(w)=\tilde u.
\end{equation*}
Next let $w_j \in B_{\rho}$ for $j=1,2$, and $\tilde u_j$ be the solution to \eqref{fixpoint} with $w$ replaced by $w_j$. Then the difference $V:=\tilde u_1 -\tilde u_2$ satisfies 
\begin{equation*}
\left\{
\begin{aligned}
&\partial_t^3 V +  (\alpha-2\beta \partial_tu_0) \partial_{t}^2 V  - \Delta_g \partial_t V-c^2\Delta_g V +\left(q-2\partial_t(\beta\partial_tu_0)\right)\partial_tV\\
&= \partial_t\bigl(\beta \partial_t(w_1-w_2)\partial_t(w_1+w_2)\bigr)  
    && \text{in } \mathcal{Q}, \\
&V= 0
    && \text{on } \Gamma, \\
&V(0,x)=0, \partial_t V(0,x) = 0, \partial_t^2 V(0,x)=0 
    && \text{in }  M.
\end{aligned}
\right.
\end{equation*}
Therefore, by using Lemma \ref{lemma2.2} we obtain 
\begin{align*}
&\| V \|_{E^{m+2}}+ \| \partial_{\nu}V \|_{N_{m+1}}  \\ 
&\leq C\| \partial_t\bigl(\beta \partial_t(w_1-w_2)\partial_t(w_1+w_2)\bigr)  \|_{E^m} \leq C \| \beta \partial_t(w_1-w_2)\partial_t(w_1+w_2) \|_{E^{m+1}} \\
& \leq C \| w_1+w_2\|_{E^{m+2}}\| w_1-w_2  \|_{E^{m+2}} \leq C \rho \|w_1-w_2\|_{E^{m+2}}.
\end{align*}
For $\rho$ sufficiently small, it implies that $\mathcal{G}$ is a contraction on $B_{\rho}$. By the Banach fixed-point theorem, there exists a unique fixed point $\bar u$ which solves \eqref{fixpoint} with $w=\bar u$. Therefore $u:=\bar u+u_0$ is the solution to \eqref{JMGTso}, and by \eqref{estimate2}, we obtain that the estimate \eqref{wellestimate}  holds.
\end{proof}

\begin{remark}
Moreover, the solution map $(f,h^{(0)},h^{(1)},h^{(2)})\,\mapsto\,u$ is $C^\infty$ Fr\'echet differentiable.
\end{remark}

\section{Geometric optics}\label{go}

In this section, we construct  geometric optics solutions  approximating  exact solutions of the linearized MGT equation 
\begin{equation}\label{linearMGT}
Pv:=\partial_t^3v + \alpha(t,x) \partial_{t}^2 v  - \Delta_g \partial_t v-c^2(x)\Delta_g v+q(t,x) \partial_t v =0
    \quad \text{in } \mathcal{Q}.
\end{equation}
\subsection{Construction of the phase and amplitude functions }
We construct geometric optics solutions of the following  WKB form
\begin{equation*}
v_{\rho}=\mathrm{e}^{\mathrm{i} \rho(t+\varphi(x))}a(t,x,\rho)= \mathrm{e}^{\mathrm{i} \rho(t+\varphi(x))}(a_0(t,x)+\rho^{-1}a_{1}(t,x)),
\end{equation*}
where $t+\varphi$ is the phase function and $a$ is the amplitude function. 
Substituting this ansatz into \eqref{linearMGT}, we have
\begin{align}\label{WKBform}
Pv_{\rho}=  \mathrm{e}^{\mathrm i \rho (t+\varphi(x))} \left( -\mathrm{i}\rho^3 T_3 a - \rho^2T_2 a+i\rho T_1 a+Pa \right),
\end{align}
where the operators $T_3,T_2$ and $T_1$ are defined by
\begin{align*}
T_3 a&=\bigl(1-|\nabla_g\varphi|_g^2\bigr)a,
\\
T_2 a&=\bigl(3-|\nabla_g\varphi|_g^2\bigr)\partial_t a
-2\langle \nabla_g\varphi,\nabla_g a\rangle_g
-\bigl(\Delta_g\varphi+c^2|\nabla_g \varphi|_g^2-\alpha\bigr)a,
\\
T_1 a&=3\partial_t^2 a
+\bigl(2\alpha-\Delta_g\varphi\bigr)\partial_t a
-2\langle \nabla_g\varphi,\nabla_g \partial_t a\rangle_g
-\Delta_g a \\
&-c^2(\Delta_g\varphi)a
-2c^2\langle \nabla_g\varphi,\nabla_g a\rangle_g
+qa.
\end{align*}
To make the leading order term in \eqref{WKBform} vanish, we impose the following eikonal equation and transport equations
\begin{align}
&|\nabla_g \varphi|_g^2=1, \label{eikonal4}  \\
 &       T_2 a_0=0          ,   \label{trans1} \\
 &T_2 a_1= \mathrm{i} T_1 a_0.           \label{trans2}
\end{align}
Similar to the construction in \cite[Section 2.1]{kian_recovery_2019} and \cite[Section 4.1]{fu2023inverse}, we now proceed to construct the phase function and the amplitude function.
Since $(M,g)$ is assumed to be simple, the eikonal equation can be solved globally in $M$. We first extend the simple manifold $(M,g)$ into a larger simple manifold $(\widetilde M,g)$ such that $M$ is contained in  the interior of $\widetilde M$. Now we choose $y \in \partial \widetilde M$ and consider the polar normal coordinates  $(r,\theta)$ on $\widetilde M$ given by
\begin{equation*}
x=\mathrm{exp}_{y}(r\theta),
\end{equation*}
where $r>0$ and $\theta \in S_{y}(\widetilde M):=\{     v \in T_y \widetilde M: |v|_g=1 \}$. According to the Gauss lemma, in these coordinates, the metric has the form $g(r,\theta)=d r^2+g_0(r,\theta)$ with $g_0(r,\theta)$ a metric on $S_y(\widetilde M)$ that depends smoothly on $r$. We choose 
\begin{equation}\label{eki}
\varphi(x)=\operatorname{dist}(y,x),     \quad x\in M,
\end{equation}
where $\operatorname{dist}$ denotes the Riemannian distance function on $(\widetilde M,g)$. Then $\varphi$ is given by $r$ in the polar normal coordinates. Since $g=dr^2+g_0(r,\theta)$, we have $\nabla_g \varphi=\partial_r$ and hence  
\begin{equation}
|\nabla_g \varphi|_g^2=g(\partial_r,\partial_r)=1.
\end{equation}
After choosing the phase function, we now construct the amplitude functions. We write $a(t,r,\theta)=a(t,\mathrm{exp}_{y} (r\theta))$ and use this notation to indicate the representation in the polar normal coordinates for other functions.  We define $b(r,\theta)=\det g_0(r,\theta)$ and $\gamma(t,r,\theta)=\alpha(t,r,\theta)-c^2(r,\theta)$. Then we transfer \eqref{trans1} and \eqref{trans2} to 
\begin{align}
&\partial_t a_0-\partial_ra_0+\left(\frac{\gamma}{2}-\frac{\partial_r b}{4b}\right)a_0=0, \label{transp1}  \\
&\partial_t a_1-\partial_r a_1+\left(\frac{\gamma}{2}-\frac{\partial_r b}{4b}\right)a_1=\frac{\mathrm{i}}{2} T_1 a_0. \label{transp2}
\end{align}
We see that for $\chi\in  C^{\infty}(\mathbb R)$ and $h\in C^{\infty}(S_y \widetilde M)$, the function 
\begin{equation}\label{pamplitude}
a_0(t,r,\theta)=\chi(r+t) h(\theta) b(r,\theta)^{-\frac{1}{4}}\mathrm{e}^{-\frac{1}{2}\int_0^t \gamma(\tau,r+t-\tau,\theta) d\tau    }
\end{equation}
is the solution of the \eqref{transp1}. Given the initial data, $a_1$  can be obtained by solving the first order ODE \eqref{transp2}. We omit the details.
\subsection{Estimate of the remainder term}
We now proceed with the construction and estimate of the remainder term.
We define $F_{\rho}=\mathrm{e}^{\mathrm{i} \rho \varphi(x)} (Pa_0+i T_1a_1+\rho^{-1} Pa_1)$, and the remainder term $R_{\rho}$ must satisfy
\begin{equation}\label{remain}
\left\{
\begin{aligned}
&\partial_t^3R_{\rho} + \alpha \partial_{t}^2 R_{\rho}  - \Delta_g \partial_t R_{\rho}-c^2\Delta_g R_{\rho}+q \partial_t R_{\rho} = -\mathrm{e}^{\mathrm{i} \rho t} F_{\rho}
    && \text{in } \mathcal{Q}, \\
&R_{\rho}= 0
    && \text{on } \Gamma, \\
&R_{\rho}(0,\cdot)=  \partial_t R_{\rho}(0,\cdot) =  \partial_t^2 R_{\rho}(0,\cdot)=0 
    && \text{in } M.
\end{aligned}
\right.
\end{equation}
We now state  the remainder estimate.
\begin{lemma}\label{WKB}
Choose $\varphi,a_0$ and $a_1$ by equations \eqref{eki}, \eqref{pamplitude}  and \eqref{transp2} respectively. Then there exists a solution $v\in E^2$ to $Pv=0$ of the form 
\begin{equation}
v=\mathrm{e}^{\mathrm{i} \rho(t+\varphi(x))}(a_0+\rho^{-1}a_1)+R_{\rho},
\end{equation}
where $R_{\rho}$ satisfies \eqref{remain}. Moreover, there exist  positive constants $C$ and $C_0$ such that 
\begin{align}
&\rho\left( \| R_{\rho} \|_{L^2(\mathcal Q)} + \|\partial_t R_{\rho} \|_{L^2(\mathcal Q)}+ \|\nabla_g R_{\rho} \|_{L^2(\mathcal Q)} \right)+\| \partial_t^2 R_{\rho} \|_{L^2(\mathcal Q)}+\| \nabla_g \partial_tR_{\rho} \|_{L^2(\mathcal Q)} \nonumber \\
&\leq C \| F_{\rho} \|_{H^1([0,T];L^2(M))} \leq C_0.\label{remainesti-new}
\end{align}
\end{lemma}
\begin{proof}
The proof follows the  arguments in \cite[Section 2.1]{kian_recovery_2019} and \cite[Section 4.2]{fu2023inverse}. 

By definition, $q\in E^{m}$ implies $q\in C([0,T];H^{m}(M))$ and $\partial_{t}q\in C([0,T];H^{m-1}(M))$. Since $m>n+1$, the Sobolev embedding theorem yields
$H^{m}(M),\,H^{m-1}(M)\hookrightarrow W^{1,\infty}(M)$,
and hence $q,\nabla_{g}q,\partial_{t}q,\nabla_{g}\partial_{t}q\in L^{\infty}(\mathcal Q)$. In particular, $q\in W^{1,\infty}(\mathcal Q)$. By the same argument, we have
$\alpha\in W^{1,\infty}(\mathcal Q)$.

We define
\[
f_\rho(x,t):=-e^{i\rho t}F_\rho(x,t).
\]

We divide the proof into two steps.

\textbf{Estimate of $\partial_t^2R_\rho$ and $\nabla_g\partial_tR_\rho$:}
Applying the basic energy estimate established in  Lemma \ref{lemma2.2} to \eqref{remain}, we obtain
\begin{equation}\label{eq:high2}
\|\partial_t^2 R_\rho\|_{L^2(\mathcal Q)}
+\|\nabla_g\partial_tR_\rho\|_{L^2(\mathcal Q)}
\le C\|f_\rho\|_{L^2(\mathcal Q)}
\le C\|F_\rho\|_{L^2(\mathcal Q)}\le C\|F_\rho\|_{H^1([0,T];L^2(M))}.
\end{equation}

\textbf{Estimate of $R_\rho$, $\partial_t R_\rho$ and $\nabla_g R_\rho$:}
Set
\[
V_\rho(x,t):=\int_0^t R_\rho(x,s)\,ds.
\]

Integrating \eqref{remain} over $(0,t)$ and using the vanishing initial data, we get
\begin{equation}\label{eq:Vrho}
\left\{
\begin{aligned}
&\partial_t^3V_\rho+\alpha\partial_t^2V_\rho-\Delta_g\partial_tV_\rho-c^2\Delta_gV_\rho+q\partial_tV_\rho
=G_\rho+K_{\rho,1}+K_{\rho,2}   &&\text{in } \mathcal Q,  \\
&V_\rho=0 \quad &&\text{on }\Gamma,\qquad  \\
&V_\rho(0,\cdot)=\partial_tV_\rho(0,\cdot)=\partial_t^2V_\rho(0,\cdot)=0  &&\text{in } M,
\end{aligned}
\right.
\end{equation}
where
\begin{align}
G_\rho(x,t)&:=\int_0^t f_\rho(x,s)\,ds
=-\int_0^t e^{i\rho s}F_\rho(x,s)\,ds,\label{eq:G}\\
K_{\rho,1}(x,t)&:=\int_0^t \partial_s\alpha(x,s)\,\partial_s^2V_\rho(x,s)\,ds,\label{eq:K1}\\
K_{\rho,2}(x,t)&:=\int_0^t \partial_sq(x,s)\,\partial_sV_\rho(x,s)\,ds.\label{eq:K3}
\end{align}

For $\tau\in[0,T]$, define
\[
\mathcal N(\tau):=
\sup_{0\le t\le \tau}
\Bigl(
\|V_\rho(t)\|_{H^2(M)}
+\|\partial_tV_\rho(t)\|_{H^1(M)}
+\|\partial_t^2V_\rho(t)\|_{L^2(M)}
\Bigr).
\]
Applying basic energy estimate to \eqref{eq:Vrho} on the time interval $[0,\tau]$, we obtain
\begin{equation}\label{eq:N1}
\mathcal N(\tau)
\le C\Bigl(
\|G_\rho\|_{L^2((0,\tau)\times M)}
+\|K_{\rho,1}\|_{L^2((0,\tau)\times M)}
+\|K_{\rho,2}\|_{L^2((0,\tau)\times M)}
\Bigr).
\end{equation}

We next estimate the two Volterra terms, $K_{\rho,1}$ and $K_{\rho,2}$. We shall use the elementary inequality
\begin{equation}\label{eq:hardy}
\int_0^\tau \Bigl\|\int_0^t h(s)\,ds\Bigr\|_X^2\,dt
\le \tau^2 \int_0^\tau \|h(t)\|_X^2\,dt
\end{equation}
for any Hilbert space $X$.

For $K_{\rho,1}$, by \eqref{eq:hardy},
\[
\|K_{\rho,1}\|_{L^2((0,\tau)\times M)}^2
\le \tau^2 \|\partial_t\alpha\|_{L^\infty(\mathcal Q)}^2
\int_0^\tau \|\partial_t^2V_\rho(t)\|_{L^2(M)}^2\,dt.
\]
Since
$\|\partial_t^2V_\rho(t)\|_{L^2(M)}\le \mathcal N(t)$,
we obtain
\begin{equation}\label{eq:K1est}
\|K_{\rho,1}\|_{L^2((0,\tau)\times M)}^2
\le C \int_0^\tau \mathcal N(t)^2\,dt.
\end{equation}

Similarly, for $K_{\rho,2}$, using \eqref{eq:hardy} and
\[
\|\partial_tV_\rho(t)\|_{L^2(M)}\le \|\partial_tV_\rho(t)\|_{H^1(M)}\le \mathcal N(t),
\]
we have
\begin{equation}\label{eq:K3est}
\|K_{\rho,2}\|_{L^2((0,\tau)\times M)}^2
\le C \int_0^\tau \mathcal N(t)^2\,dt.
\end{equation}

Substituting \eqref{eq:K1est} and \eqref{eq:K3est} into \eqref{eq:N1}, we obtain
\begin{equation}\label{eq:N2}
\mathcal N(\tau)
\le C\|G_\rho\|_{L^2((0,\tau)\times M)}
+ C\Bigl(\int_0^\tau \mathcal N(t)^2\,dt\Bigr)^{1/2}.
\end{equation}
Squaring both sides and using $(a+b)^2\le 2a^2+2b^2$, we get
\begin{equation}\label{eq:N3}
\mathcal N(\tau)^2
\le C\|G_\rho\|_{L^2((0,\tau)\times M)}^2
+ C\int_0^\tau \mathcal N(t)^2\,dt.
\end{equation}

We now estimate $G_\rho$. Integrating by parts in time in \eqref{eq:G}, we obtain
\[
G_\rho(x,t)
=
-\frac{1}{i\rho}
\left(
e^{i\rho t}F_\rho(x,t)-F_\rho(x,0)-\int_0^t e^{i\rho s}\partial_sF_\rho(x,s)\,ds
\right).
\]
Therefore,
\begin{equation}\label{eq:Gest}
\|G_\rho\|_{L^2((0,\tau)\times M)}
\le C\rho^{-1}\|F_\rho\|_{H^1([0,\tau];L^2(M))}
\le C\rho^{-1}\|F_\rho\|_{H^1([0,T];L^2(M))},
\end{equation}
where we used the trace estimate at $t=0$ and the Cauchy--Schwarz inequality.

Combining \eqref{eq:N3} and \eqref{eq:Gest}, and then applying Gronwall's inequality, we obtain
\begin{equation}\label{eq:N4}
\mathcal N(\tau)\le C\rho^{-1}\|F_\rho\|_{H^1([0,T];L^2(M))}
\qquad \text{for all } \tau\in[0,T].
\end{equation}
In particular, taking $\tau=T$ and recalling that
\[
R_\rho=\partial_tV_\rho,\qquad
\partial_tR_\rho=\partial_t^2V_\rho,\qquad
\nabla_g R_\rho=\nabla_g\partial_tV_\rho,
\]
we deduce
\begin{align}
\|R_\rho\|_{L^2(\mathcal Q)}
&\le T^{1/2}\sup_{0\le t\le T}\|\partial_tV_\rho(t)\|_{L^2(M)}
\le CT^{1/2}\mathcal N(T),\label{eq:R0}\\
\|\partial_tR_\rho\|_{L^2(\mathcal Q)}
&\le T^{1/2}\sup_{0\le t\le T}\|\partial_t^2V_\rho(t)\|_{L^2(M)}
\le CT^{1/2}\mathcal N(T),\label{eq:R1}\\
\|\nabla_g R_\rho\|_{L^2(\mathcal Q)}
&\le T^{1/2}\sup_{0\le t\le T}\|\nabla_g\partial_tV_\rho(t)\|_{L^2(M)}
\le CT^{1/2}\mathcal N(T).\label{eq:R2}
\end{align}
Hence, by \eqref{eq:N4},
\begin{equation}\label{eq:low}
\|R_\rho\|_{L^2(\mathcal Q)}
+\|\partial_tR_\rho\|_{L^2(\mathcal Q)}
+\|\nabla_g R_\rho\|_{L^2(\mathcal Q)}
\le C\rho^{-1}\|F_\rho\|_{H^1([0,T];L^2(M))}.
\end{equation}

Finally, combining \eqref{eq:low} with \eqref{eq:high2}, we obtain
\begin{align*}
&\rho\Bigl(\|R_\rho\|_{L^2(\mathcal Q)}
+\|\partial_tR_\rho\|_{L^2(\mathcal Q)}
+\|\nabla_g R_\rho\|_{L^2(\mathcal Q)}\Bigr)
+\|\partial_t^2R_\rho\|_{L^2(\mathcal Q)}
+\|\nabla_g\partial_tR_\rho\|_{L^2(\mathcal Q)} \\
&\le C\|F_\rho\|_{H^1([0,T];L^2(M))}.
\end{align*}

Since $\varphi$, $a_0$, and $a_1$ are independent of $\rho$ and have the required finite regularity, it follows from the definition of $F_\rho$ that there exists a constant $C_0>0$, independent of $\rho$, such that 
\[
\|F_\rho\|_{H^1([0,T];L^2(M))}\le C_0.
\]
This completes the proof.
\end{proof}

\begin{remark}\label{rema}
By an analogous argument, the adjoint equation 
\begin{equation}
P^{\star} v:=-\partial_t^3v+\partial_t^2(\alpha v)+\Delta_g\partial_tv-\Delta_g(c^2v)-\partial_t(qv)=0
\end{equation}
admits a solution of the form
\begin{equation}
v=\mathrm{e}^{-\mathrm{i} \rho(t+\varphi(x))}(a_0+\rho^{-1}a_1)+R_{\rho},
\end{equation}
where $\varphi(x)=\operatorname{dist}(y,x)$ with $y \in \partial \widetilde M$ and $a_0(t,r,\theta)=\chi(r+t) h(\theta)b(r,\theta)^{-\frac{1}{4}}\mathrm{e}^{\frac{1}{2}\int_0^t \gamma(\tau,r+t-\tau,\theta) d\tau    }$.
Moreover, there exists a constant $C_1$ such that 
\begin{align}
&\rho\left( \| R_{\rho} \|_{L^2(\mathcal Q)} + \|\partial_t R_{\rho} \|_{L^2(\mathcal Q)}+ \|\nabla_g R_{\rho} \|_{L^2(\mathcal Q)} \right)+\| \partial_t^2 R_{\rho} \|_{L^2(\mathcal Q)}+\| \nabla_g \partial_tR_{\rho} \|_{L^2(\mathcal Q)} 
 \leq C_1.\label{remainestim}
\end{align}
\end{remark}
\section{Unique recovery of lower-order coefficients in the MGT equation}\label{uniq}

In this section, we prove the unique recovery of $(\alpha,q)$ from the all boundary measurement map $\mathcal L_{\alpha,q}$ in the linearized MGT equation. The linearized MGT equation is as follows:
\begin{equation}\label{MGTso}
\left\{
\begin{aligned}
&\partial_t^3v + \alpha \partial_{t}^2 v  - \Delta_g \partial_t v-c^2\Delta_g v+q \partial_t v = 0
    && \text{in } \mathcal{Q}, \\
&v= f
    && \text{on } \Gamma, \\
&v(0,\cdot)=h^{(0)},  \ \partial_t v(0,\cdot) = h^{(1)},\ \partial_t^2 v(0,\cdot)=h^{(2)} 
    && \text{in } M.
\end{aligned}
\right.
\end{equation}
We first define the all boundary measurement map for the MGT equation \eqref{MGTso}
\begin{equation*}
\mathcal L_{\alpha,q}: \big(f,h^{(0)},h^{(1)},h^{(2)}\big) \to \bigl((\partial_{\nu}\partial_t v +c^2\partial_{\nu }v)|_{\Gamma},v(T,\cdot),\partial_t v(T,\cdot),\partial_t^2 v(T,\cdot) \bigr).
\end{equation*}
\begin{theorem}\label{unirecover}
Assume that $\alpha_j,q_j \in E_{0,m-2}^m$ for $j=1,2$, and $c\in C^{\infty} (M)$.
Let $\mathcal L_{\alpha_j,q_j}$ denote the corresponding all boundary measurement map for the MGT equation \eqref{MGTso} with
$(\alpha,q)$
replaced by
$(\alpha_j,q_j)$, for $j=1,2$.
Suppose that
\[
\mathcal L_{\alpha_1,q_1}\big(f,h^{(0)},h^{(1)},h^{(2)}\big)
=
\mathcal L_{\alpha_2,q_2}\big(f,h^{(0)},h^{(1)},h^{(2)}\big),
\qquad \forall \big(f,h^{(0)},h^{(1)},h^{(2)}\big)\in  R.
\]
Then we have
\[
\alpha_2=\alpha_1, \quad
q_2=q_1  \quad \text{ in } \mathcal Q.
\]
\end{theorem}

\begin{proof}
Denote $\tilde v=v_1-v_2$. Then $\tilde v$ satisfies the equation
\begin{equation}\label{MGTminus}
\left\{
\begin{aligned}
&\partial_t^3\tilde v + \alpha_1 \partial_{t}^2 \tilde v  - \Delta_g \partial_t \tilde v-c^2\Delta_g \tilde v+q_1 \partial_t \tilde v = \alpha \partial_t^2v_2+q \partial_t v_2
    && \text{in } \mathcal{Q}, \\
&\tilde v= 0
    && \text{on } \Gamma, \\
&\tilde v(0,\cdot)= \partial_t \tilde v(0,\cdot) = \partial_t^2 \tilde v(0,\cdot)=0 
    && \text{in } M,
\end{aligned}
\right.
\end{equation}
where $\alpha=\alpha_2-\alpha_1$ and $q=q_2-q_1$. Moreover, the equality $\mathcal L_{\alpha_1,q_1}=\mathcal L_{\alpha_2,q_2}$ implies that $\tilde v$ satisfies the boundary condition $\partial_{\nu} \partial_t \tilde v+c^2 \partial_{\nu} \tilde v=0$ on $\Gamma$ and final condition $\tilde v(T,\cdot)=\partial_t \tilde v(T,\cdot)=\partial_t^2\tilde v(T,\cdot)=0$ in $M$. 

Let $v_0$ be a solution of the adjoint equation 
\begin{equation*}
-\partial_t^3v_0+\partial_t^2(\alpha_1 v_0)+\Delta_g\partial_tv_0-\Delta_g(c^2v_0)-\partial_t(q_1v_0)=0.
\end{equation*}
Then we derive an integral identity
\begin{align}
&\int_0^T \int_M  (\alpha \partial_t^2 v_2+q\partial_t v_2)v_0 dV_g dt \nonumber \\
&=\int_0^T \int_M(\partial_t^3\tilde v + \alpha_1 \partial_{t}^2 \tilde v  - \Delta_g \partial_t \tilde v-c^2\Delta_g \tilde v+q_1 \partial_t \tilde v )v_0 dV_g dt\nonumber \\
&=\int_0^T\int_M\tilde v(-\partial_t^3v_0+\partial_t^2(\alpha_1 v_0)+\Delta_g\partial_tv_0-\Delta_g(c^2v_0)-\partial_t(q_1v_0)) \nonumber \\
&+(\partial_t \tilde v \Delta_g v_0-v_0\Delta_g \partial_t \tilde v )+(\tilde v \Delta_g(c^2 v_0)-c^2 v_0\Delta_g \tilde v) dV_g dt \nonumber \\
&=\int_0^T\int_{\partial M}(\partial_t \tilde v \partial_{\nu} v_0+\tilde v \partial_{\nu}(c^2 v_0)-v_0(\partial_{\nu} \partial_{t} \tilde v+c^2 \partial_{\nu} \tilde v)) dS_g dt  \nonumber \\
&=0.  \label{ideM}
\end{align}
Using Lemma \ref{WKB} and Remark \ref{rema}, we choose solutions of the form
\begin{align}
v_2=\mathrm{e}^{\mathrm{i}\rho(t+\varphi)}(a_{2,0}+\rho^{-1}a_{2,1})+R_{2,\rho}, \ v_0=\mathrm{e}^{-\mathrm{i}\rho(t+\varphi)}(a_{0,0}+\rho^{-1}a_{0,1})+R_{0,\rho},
\end{align}
where the remainder terms $R_{j,\rho}$ satisfy the estimate that 
\begin{equation}
\rho\left(\|R_{j,\rho}\|_{L^2 (\mathcal Q)}+\| \partial_t R_{j,\rho} \|_{L^2(\mathcal Q)}\right)+\| \partial_t^2 R_{j,\rho}\|_{L^2(\mathcal Q)}\le C,
\end{equation}
for $j=0,2$.
We choose the principal amplitude functions $a_{j,0}$, $j=0,2$,  defined in the polar coordinates associated with $y\in \partial \widetilde M$ by 
\begin{equation}
a_{2,0}=\chi(r+t)h(\theta)b(r,\theta)^{-\frac{1}{4}} \mathrm{e}^{-\frac{1}{2}\int_0^t \gamma_2(\tau,r+t-\tau,\theta)d\tau}, \ \ a_{0,0}=b(r,\theta)^{-\frac{1}{4}} \mathrm{e}^{\frac{1}{2} \int_0^t \gamma_1(\tau,r+t-\tau,\theta)d\tau},
\end{equation}
where $\chi$ is an arbitrary smooth function on $\mathbb{R}$ and $h$ is an arbitrary smooth function on the unit sphere at $y$. 
Substituting  geometric optics solutions into  \eqref{ideM},  and multiplying by $-\rho^{-2}$, we obtain, by letting $\rho\to\infty$,
\begin{equation}
\lim_{\rho\to\infty} -\rho^{-2}\int_0^T \int_M  \alpha \partial_t^2 v_2 v_0 +q\partial_t v_2 v_0 dV_g dt=\int_0^T\int_M \alpha a_{2,0} a_{0,0} dV_g dt=0.
\end{equation}
Since $\gamma:=\gamma_2-\gamma_1=(\alpha_2-c^2)-(\alpha_1-c^2)=\alpha$ in $\mathcal Q$,
we choose zero extensions
$\widetilde\alpha,\widetilde\gamma \in L^\infty(\mathbb R\times \widetilde M)$. Then we obtain 
$\widetilde\alpha= \widetilde\gamma$ in  $\mathbb R\times \widetilde M$.

 Denoting by $\tau_+(y,\theta)$ the time of existence in $\widetilde M$ of the maximal geodesic $\sigma_{y,\theta}$ satisfying $\sigma_{y,\theta}(0)=y$ and $\sigma'_{y,\theta}(0)=\theta$,
we obtain in the polar normal coordinates,
\begin{equation}\label{eq:main-identity-gamma}
\int_{0}^{+\infty}\int_{S_y\widetilde M}\int_{0}^{\tau_+(y,\theta)}
\widetilde \gamma(t,\sigma_{y,\theta}(r))\,h(\theta)\,\chi(r+t)\,
\mathrm{e}^{-\frac{1}{2}
\int_{0}^{t}\widetilde\gamma(\tau,\sigma_{y,\theta}(r+t-\tau))\,d\tau
}\,dr\,d\theta\,dt=0,
\end{equation}
for every $y\in \partial \widetilde M$, every $h\in C^\infty(S_y\widetilde M)$, and every $\chi\in C_c^\infty(\mathbb R)$. Here we use the fact that $dV_g=b(r,\theta)^{\frac{1}{2}}dr d\theta$  in the polar normal coordinates.

We next utilize the integral identity  to prove that
\[
\widetilde\gamma\equiv 0 \qquad \text{in } (0,T)\times \widetilde M.
\]

Fix $y\in\partial \widetilde M$ and denote $\sigma_{y,\theta}(r)=(r,\theta)$. Since \eqref{eq:main-identity-gamma} holds for all $h\in C^\infty(S_y \widetilde M)$, it follows that for every $\theta\in S_y \widetilde M$,
\begin{equation}\label{eq:thetawise-identity}
\int_{0}^{+\infty}\int_{0}^{\tau_+(y,\theta)}
\widetilde\gamma(t,\sigma_{y,\theta}(r))\,\chi(r+t)\,
\mathrm{e}^{
-\frac{1}{2} \int_{0}^{t}\widetilde \gamma(\tau,\sigma_{y,\theta}(r+t-\tau))\,d\tau} dr\,dt=0.
\end{equation}

We define the inward pointing boundary of unit sphere bundle $\partial_+ S\widetilde M= \{ (x,\theta)\in S \widetilde M: x\in \partial \widetilde M, \langle \theta, \nu(x)\rangle_g<0  \}$.
Now fix $(y,\theta)\in\partial_{+} S \widetilde M$ and write $\tau_+:=\tau_+(y,\theta)$. 
Introduce the change of variables
\[
s=t+r.
\]
Then $t=s-r$, and the integration region
\[
t\ge 0,\qquad 0\le r\le \tau_+
\]
becomes
\[
s\ge 0,\qquad 0\le r\le \min\{s,\tau_+\}.
\]
Hence \eqref{eq:thetawise-identity} can be rewritten as
\[
\int_0^{+\infty} \chi(s)\,B(s,y,\theta)\,ds=0,
\]
where
\begin{equation}\label{eq:def-B}
B(s,y,\theta)
:=
\int_0^{\min\{s,\tau_+\}}
\widetilde \gamma(s-r,\sigma_{y,\theta}(r))
\mathrm{e}^{ 
-\frac{1}{2}\int_0^{s-r} \widetilde\gamma(\tau,\sigma_{y,\theta}(s-\tau))\,d\tau
} dr .
\end{equation}
Since $\chi\in C_c^\infty(\mathbb R)$ is arbitrary, we obtain
\begin{equation}\label{distri}
B(\cdot,y,\theta)=0 \qquad \text{in } \mathcal D'(0,\infty).
\end{equation}

We now compute $B(s,y,\theta)$ more explicitly. Set
\[
a(s):=\max\{0,s-\tau_+\}.
\]
Using the substitution $t=s-r$, formula \eqref{eq:def-B} becomes
\[
B(s,y,\theta)
=
\int_{a(s)}^{s}
\widetilde \gamma(t,\sigma_{y,\theta}(s-t))
\mathrm{e}^{
-\frac{1}{2}\int_0^{t}\widetilde\gamma(\tau,\sigma_{y,\theta}(s-\tau))\,d\tau
}dt.
\]
For fixed $(s,y,\theta)$, define
\[
E_s(t):=
\mathrm{e}^{ 
-\frac{1}{2}\int_0^{t}\widetilde\gamma(\tau,\sigma_{y,\theta}(s-\tau))\,d\tau
}.
\]
Then
\[
E_s'(t)
=
-\frac{1}{2}\widetilde \gamma(t,\sigma_{y,\theta}(s-t))\,E_s(t),
\]
and hence
\begin{align*}
B(s,y,\theta)
&=
-2\int_{a(s)}^{s}E_s'(t)\,dt
=
-2(E_s(s)-E_s(a(s))) \\
&=
2\mathrm{e}^{
-\frac{1}{2}\int_0^{a(s)}\widetilde\gamma(\tau,\sigma_{y,\theta}(s-\tau))\,d\tau
}
\left(
1-
\mathrm{e}^{
-\frac{1}{2}\int_{a(s)}^{s}\widetilde\gamma(\tau,\sigma_{y,\theta}(s-\tau))\,d\tau
}
\right).
\end{align*}
Since the function $B(\cdot,y,\theta)$ is continuous in the variable $s$, we have $B(s,y,\theta)=0$ for all $s>0$.
 Using that $\widetilde\gamma$ is real-valued, we deduce
\[
\int_{a(s)}^{s}\widetilde\gamma(t,\sigma_{y,\theta}(s-t))\,dt=0,
\qquad \forall  s>0.
\]
Returning to the variable $r=s-t$, this is
\[
\int_0^{\min\{s,\tau_+\}}
\widetilde\gamma(s-r,\sigma_{y,\theta}(r))\,dr=0,
\qquad \forall  s>0.
\]
Since $\widetilde\gamma(t,\cdot)=0$ for $t<0$, the last identity can be written uniformly as
\begin{equation}\label{eq:light-ray-zero}
\int_0^{\tau_+(y,\theta)}
\widetilde\gamma(s-r,\sigma_{y,\theta}(r))\,dr=0,
\qquad \forall  s\in \mathbb R,\ \forall (y,\theta)\in \partial_+S \widetilde M.
\end{equation}
Therefore the light ray transform of $\widetilde\gamma$ vanishes.

Next, take the Fourier transform of \eqref{eq:light-ray-zero} in the $s$-variable. Denote
\[
\widehat{\widetilde\gamma}(\lambda,x)
:=
\int_{\mathbb R}e^{-i\lambda t}\widetilde\gamma(t,x)\,dt.
\]
By Fubini's theorem, for every $\lambda\in \mathbb R$,
\[
0
=
\int_{\mathbb R}e^{-i\lambda s}
\left(
\int_0^{\tau_+(y,\theta)}
\widetilde\gamma(s-r,\sigma_{y,\theta}(r))\,dr
\right)\,ds
=
\int_0^{\tau_+(y,\theta)}
e^{-i\lambda r}\,
\widehat{\widetilde\gamma}\bigl(\lambda,\sigma_{y,\theta}(r)\bigr)\,dr.
\]
Therefore
\[
I_{i\lambda}\bigl(\widehat{\widetilde\gamma}(\lambda,\cdot)\bigr)(y,\theta)=0,
\qquad \forall (y,\theta)\in \partial_+S\widetilde M,\ \forall \lambda\in \mathbb R.
\]
We regard \(I_{i\lambda}\) as a weighted geodesic ray transform associated with the same geodesic family, with weight
\[
w_\lambda(x,\xi)=e^{-i\lambda \tau_-(x,\xi)},
\]
so that along each geodesic \(\sigma_{y,\theta}\),
\[
w_\lambda(\sigma_{y,\theta}(r),\dot\sigma_{y,\theta}(r))=e^{-i\lambda r}.
\]
Taking $w_0\equiv 1$ as the reference weight, the  geodesic ray
transform is injective  on the simple manifold $(\widetilde M,g)$.
Moreover, $w_\lambda\to w_0$ in $C^2$ as $\lambda\to0$.
Applying the perturbation result for weighted ray transforms
\cite[Theorem 2.3(b)]{frigyik2008x} to the full geodesic family on
a slightly larger simple extension of $(\widetilde M,g)$, we obtain that the
corresponding weighted geodesic ray transform $I_{i\lambda}$ remains injective
for all sufficiently small $|\lambda|$.
Hence there exists \(\varepsilon>0\) such that
\[
\widehat{\widetilde\gamma}(\lambda,\cdot)=0,
\qquad |\lambda|<\varepsilon.
\]

Finally, since $\widetilde\gamma$ is  compactly supported in the time variable, for every fixed $x\in \widetilde M$, the function
\[
z\mapsto \widehat{\widetilde\gamma}(z,x)
:=
\int_{\mathbb R} e^{-izt}\widetilde\gamma(t,x)\,dt,
\qquad z\in\mathbb C,
\]
is entire. As it vanishes on a real interval, whose points have accumulation points in \(\mathbb C\), it must vanish identically. Hence
\[
\widehat{\widetilde\gamma}(\lambda,x)\equiv 0
\qquad \forall \lambda\in \mathbb R,\ \forall x\in \widetilde M.
\]
By Fourier inversion,
\[
\widetilde\gamma(t,x)\equiv 0
\qquad \text{in } \mathbb R\times \widetilde M.
\]
Restricting to $(0,T)\times  M$, we obtain
\[
\gamma_1= \gamma_2,\quad  \alpha_1=\alpha_2.
\]
After substituting \(\alpha=0\) into the integral identity, the order-\(\rho\) term yields
\begin{equation}
\lim_{\rho \to +\infty} -\mathrm{i}\rho^{-1} \int_0^T \int_M q \partial_t v_2 v_0 dV_g dt =\int_0^T \int_M q a_{2,0}a_{0,0} dV_g dt=0.
\end{equation}
In the polar normal coordinates, this becomes 
\begin{equation}
\int_0^{+\infty}\int_{S_y \widetilde M} \int_0^{\tau_+(y,\theta)}q(t,r,\theta) h(\theta)\chi(r+t) dr d\theta dt=0.
\end{equation}
Repeating the argument above, we can conclude that $q(t,x)=0$ in $\mathcal Q$.
This completes the proof.

\end{proof}
\section{Proof of Theorem \ref{thme2}}\label{proo}

In this section, we provide the proof of Theorem \ref{thme2} and the proof of Corollary \ref{cor1.3}.
We divide the proof of Theorem \ref{thme2} into several steps. First, by applying first order linearization method, we show that the nonlinear equation induces the same all boundary measurement map for the corresponding linearized MGT equation. Then, by Theorem \ref{unirecover}, we uniquely recover the linear damping coefficients $(\tilde\alpha,\tilde q)$, where
\[
\tilde\alpha=\alpha-2\beta\partial_t u_0,
\qquad
\tilde q=q-2\partial_t(\beta\partial_t u_0).
\]
Next, using the second order linearization, we derive an integral identity for the nonlinear coefficient $\beta$. Finally, by inserting suitable geometric optics solutions into this identity, we obtain the unique recovery of $\beta$. This  yields the recovery of $\alpha$, $q$, and $F$ up to the gauge symmetry.
\begin{proof}
For $(f_j,h^{(0)}_j,h^{(1)}_j,h^{(2)}_j)\in R$, $j=1,2$, we choose $\varepsilon_1$ and $\varepsilon_2$ sufficiently small so  that 
\begin{equation}
(f,h^{(0)},h^{(1)},h^{(2)})=(f_0,h^{(0)}_0,h^{(1)}_0,h^{(2)}_0)+\varepsilon_1(f_1,h^{(0)}_1,h^{(1)}_1,h^{(2)}_1)+\varepsilon_2(f_2,h^{(0)}_2,h^{(1)}_2,h^{(2)}_2)\in G_{\varepsilon},
\end{equation}
where $G_{\varepsilon}$ is defined in \eqref{Gep}. By the well-posedness result in Theorem \ref{prop1.1}, there exists a unique solution $u_j$ to \eqref{JMGTso} with $(\alpha,q,\beta,F)$ replaced by $(\alpha_j,q_j,\beta_j,F_j)$ with boundary and initial data $(f,h^{(0)},h^{(1)},h^{(2)})$, near the solution $u_{0,j}$. Denote
\begin{equation}
v_{j,k}=\frac{\partial u_j}{\partial \varepsilon_k}\bigg|_{\varepsilon_1=\varepsilon_2=0}, \quad  w_j=\frac{\partial^2 u_j}{\partial \varepsilon_1 \partial \varepsilon_2}\bigg|_{\varepsilon_1=\varepsilon_2=0}.
\end{equation}
Applying  $\frac{\partial}{\partial \varepsilon_1}\big|_{\varepsilon_1=\varepsilon_2=0}$ to equation \eqref{JMGTso}, we obtain
\begin{equation}\label{firstlinear}
\left\{
\begin{aligned}
&\partial_t^3v_{j,1} + \tilde\alpha_j \partial_{t}^2 v_{j,1}  - \Delta_g \partial_t v_{j,1}-c^2\Delta_g v_{j,1}+\tilde q_j \partial_t v_{j,1} = 0
    && \text{in } \mathcal{Q}, \\
&v_{j,1}= f_1
    && \text{on } \Gamma, \\
&v_{j,1}(0,\cdot)=h^{(0)}_1,  \ \partial_t v_{j,1}(0,\cdot) = h^{(1)}_1,\ \partial_t^2 v_{j,1}(0,\cdot)=h^{(2)}_1 
    && \text{in } M,
\end{aligned}
\right.
\end{equation}
where $\tilde \alpha_j=\alpha_j-2\beta_j \partial_t u_{0,j}$ and $\tilde q_j=q_j-2\partial_t(\beta_j\partial_t u_{0,j})$.

Observe that 
\begin{equation}
\mathcal L_{\tilde \alpha_j,\tilde q_j}^{\textrm{linear}}\big(f_1,h^{(0)}_1,h^{(1)}_1,h^{(2)}_1\big)=\frac{\partial}{\partial \varepsilon_1}\Big(   \mathcal L_{\alpha_j,q_j,\beta_j,F_j}\big(f,h^{(0)},h^{(1)},h^{(2)}\big)   \Big)\Big|_{\varepsilon_1=\varepsilon_2=0}.
\end{equation}
Then we have 
\begin{equation}
\mathcal L_{\tilde \alpha_1,\tilde q_1}^{\textrm{linear}}\big(f_1,h^{(0)}_1,h^{(1)}_1,h^{(2)}_1\big)=\mathcal L_{\tilde \alpha_2,\tilde q_2}^{\textrm{linear}}\big(f_1,h^{(0)}_1,h^{(1)}_1,h^{(2)}_1\big),  \quad   \forall(f_1,h^{(0)}_1,h^{(1)}_1,h^{(2)}_1) \in R. 
\end{equation}
Using Theorem \ref{unirecover}, we obtain
\begin{equation}\label{relate}
\tilde \alpha_1=\tilde \alpha_2:=\tilde \alpha, \quad \tilde q_1=\tilde q_2:=\tilde q.
\end{equation}
By the well-posedness result for the linear MGT equation, Lemma \ref{lemma2.2}, we have
\begin{equation}
v_{1,1}=v_{2,1}:=v_1, \quad v_{1,2}=v_{2,2}=v_{2}.
\end{equation}
Applying $\frac{\partial^2}{\partial \varepsilon_1 \partial \varepsilon_2} \big|_{\varepsilon_1=\varepsilon_2=0}$ to equation \eqref{JMGTso}, we can see that
\begin{equation}\label{secondlinear}
\left\{
\begin{aligned}
&\partial_t^3 w_j + \tilde\alpha \partial_{t}^2 w_j  - \Delta_g \partial_t w_j-c^2\Delta_g w_j+\tilde q \partial_t w_j = 2\partial_t \left(\beta_j \partial_t v_1 \partial_t v_2  \right) 
    && \text{in } \mathcal{Q}, \\
&w_j= 0
    && \text{on } \Gamma, \\
&w_j(0,\cdot)= \partial_t w_j(0,\cdot) = \partial_t^2 w_j(0,\cdot)=0 
    && \text{in } M.
\end{aligned}
\right.
\end{equation}
Let $v_0$ solve the adjoint equation
\begin{equation}\label{adj}
-\partial_t^3v_0+\partial_t^2(\tilde \alpha v_0)+\Delta_g\partial_tv_0-\Delta_g(c^2v_0)-\partial_t(\tilde q v_0)=0    \quad \text{ in } \mathcal Q.
\end{equation}
Set
\begin{equation*}
w=w_2-w_1, \quad \delta\beta=\beta_2-\beta_1.
\end{equation*}
Note that
\begin{align*}
&\frac{\partial^2}{\partial \varepsilon_1 \partial \varepsilon_2}\Big( \mathcal{L}_{\alpha_j,q_j,\beta_j,F_j}\big(f,h^{(0)},h^{(1)},h^{(2)} \big)   \Big) \big|_{\varepsilon_1 =\varepsilon_2=0}\\ &=   \Big( \big( \partial_{\nu} \partial_t w_j+c^2 \partial_{\nu} w_j\big)\big|_{\Gamma}, w_j(T,\cdot),\partial_t w_j(T,\cdot),\partial_t^2 w_j(T,\cdot)  \Big),
\end{align*}
is determined uniquely by $\mathcal L_{\alpha_j,q_j,\beta_j,F_j}$. Then the same all boundary measurement map  implies $\partial_{\nu}\partial_t w+c^2 \partial_{\nu}w=0$ on $\Gamma$ and $w(T,\cdot)=\partial_t w(T,\cdot)=\partial_t^2 w(T,\cdot)=0$ in $M$.

Using integration by parts, we obtain
\begin{align}
& \int_0^T \int_M 2 \partial_t(\delta\beta \partial_tv_1 \partial_t v_2) v_0 \ dV_g dt \nonumber \\
&=\int_0^T\int_M (\partial_t^3 w+ \tilde \alpha \partial_t^2w- \Delta_g \partial_t w-c^2\Delta_g w+\tilde q \partial_t w)v_0 dV_g dt \nonumber \\
&=\int_0^T \int_Mw(-\partial_t^3v_0+\partial_t^2(\tilde \alpha v_0)+\Delta_g\partial_tv_0-\Delta_g(c^2v_0)-\partial_t(\tilde q v_0)) \nonumber \\
&+( \partial_t w \Delta_gv_0-v_0\Delta_g \partial_t w)+(w\Delta_g(c^2 v_0)-c^2 v_0 \Delta_g w)  \ dV_g dt \nonumber \\
&=\int_0^T \int_{\partial M} (\partial_tw \partial_{\nu} v_0- v_0\partial_{\nu}\partial_t w)+(w\partial_{\nu}(c^2 v_0)-c^2 v_0 \partial_{\nu}  w) \ dS_gdt \nonumber \\
&=\int_0^T \int_{\partial M} -v_0(\partial_{\nu} \partial_t w+c^2 \partial_{\nu} w) dS_g dt=0. \label{secondide}
\end{align}
We choose 
\begin{align}
&v_2=\mathrm{e}^{\mathrm{i} \rho(t+\varphi)}(a_{2,0}+\rho^{-1}a_{2,1})+R_{2,\rho}, \\
&v_0=\mathrm{e}^{-\mathrm{i} \rho (t+\varphi)}(a_{0,0}+\rho^{-1} a_{0,1})+R_{0,\rho},
\end{align}
with 
\begin{equation*}
\rho\left(\|R_{j,\rho}\|_{L^2 (\mathcal Q)}+\| \partial_t R_{j,\rho} \|_{L^2(\mathcal Q)}\right)+\| \partial_t^2 R_{j,\rho}\|_{L^2(\mathcal Q)}\le C, \ \text{ for } j=0,2.
\end{equation*}
Inserting these geometric optics solutions into the integral identity \eqref{secondide}, multiplying by $-\rho^{-2}$ and letting $\rho\to+\infty$, we obtain
\begin{equation}
\lim_{\rho \to +\infty} - \rho^{-2} \int_0^T \int_M 2\partial_t(\delta\beta\partial_t v_1 \partial_t v_2) v_0 dV_g dt=\int_0^T\int_M 2\delta\beta \partial_t v_1 a_{2,0}a_{0,0} dV_g dt=0.
\end{equation}
We choose $a_{2,0}$ and $a_{0,0}$ as
\begin{equation*}
a_{2,0}=\chi(r+t)h(\theta)b(r,\theta)^{-\frac{1}{4}} \mathrm{e}^{-\frac{1}{2}\int_0^t (\tilde \alpha-c^2)(\tau,r+t-\tau,\theta)d\tau}, \  a_{0,0}=b(r,\theta)^{-\frac{1}{4}} \mathrm{e}^{\frac{1}{2} \int_0^t (\tilde \alpha-c^2)(\tau,r+t-\tau,\theta)d\tau}.
\end{equation*}
Regarding $\delta\beta  \partial_t v_1$ as the unknown function and repeating the  argument used in the proof of Theorem \ref{unirecover}, we obtain
\[
\delta\beta \partial_t v_1=0 \qquad \text{in } \mathcal Q.
\]
 Taking an arbitrary solution $\mathcal{V}_0$ to  the adjoint equation \eqref{adj}, we obtain
\[
\int_0^T\int_M \delta\beta\, \partial_t v_1 \mathcal{V}_0\, dV_gdt=0.
\]
Since $v_1$ is an arbitrary solution to \eqref{firstlinear}, we repeat once again the argument in the proof of Theorem \ref{unirecover}, with $\delta \beta$ in place of the unknown function, we conclude that
\[
\delta\beta=0 \qquad \text{in } \mathcal Q.
\]
Hence
$\beta_1=\beta_2$ in $\mathcal Q$.

We denote 
\begin{equation*}
\psi=u_{0,2}-u_{0,1} \in E^{m+2},
\end{equation*}
with $\psi|_{\Gamma}=(\partial_{\nu} \partial_t \psi+c^2 \partial_{\nu} \psi)|_{\Gamma}=0$, $\psi(0,\cdot)=\partial_t \psi(0,\cdot)=\partial_t^2 \psi(0,\cdot)=0$ in $M$ and $\psi(T,\cdot)=\partial_t \psi(T,\cdot)=\partial_t^2\psi(T,\cdot)=0$ in $M$.
Then using \eqref{relate}, we obtain
\begin{equation}\label{rela}
\alpha_1=\alpha_2-2\beta \partial_t \psi, \quad  q_1=q_2-2\partial_t(\beta \partial_t \psi)  \quad \text{ in } \mathcal Q.
\end{equation}
Using \eqref{rela} and the identity $u_{0,2}=u_{0,1}+\psi$, a direct calculation gives
\begin{align*}
&\partial_t^3 (u_{0,1}+\psi)+\alpha_1\partial_t^2(u_{0,1}+\psi)-\Delta_g \partial_t(u_{0,1}+\psi)-c^2\Delta_g (u_{0,1}+\psi)+q_1\partial_t(u_{0,1}+\psi) \\
&=\partial_t^3\psi+\alpha_1\partial_t^2\psi-\Delta_g\partial_t\psi-c^2\Delta_g\psi
+q_1\partial_t\psi +F_1+\partial_t (\beta(\partial_t u_{0,1})^2)\\
&=\partial_t^3\psi+\alpha_1\partial_t^2\psi-\Delta_g\partial_t\psi-c^2\Delta_g\psi
+q_1\partial_t\psi+F_1+\partial_t(\beta(\partial_t(u_{0,2}-\psi))^2) \\
&=\partial_t^3\psi+\alpha_1\partial_t^2\psi-\Delta_g\partial_t\psi-c^2\Delta_g\psi
+q_1\partial_t\psi+\partial_t(\beta(\partial_t \psi)^2)+F_1+\partial_t(\beta(\partial_t u_{0,2})^2) \\
&-2\beta \partial_t \psi\partial_t^2 u_{0,2}-2\partial_t(\beta \partial_t \psi) \partial_t u_{0,2},
\end{align*}
we finally obtain that 
\begin{align*}
\partial_t^3 u_{0,2}+ \alpha_2 \partial_t^2 u_{0,2}-\Delta_g \partial_t u_{0,2}-c^2 \Delta_g u_{0,2} +q_2 \partial_t u_{0,2}=\partial_t(\beta(\partial_t u_{0,2})^2)  \\
+F_1+\partial_t^3\psi+\alpha_1\partial_t^2\psi-\Delta_g\partial_t\psi-c^2\Delta_g\psi
+q_1\partial_t\psi+\partial_t(\beta(\partial_t \psi)^2),
\end{align*}
and hence
\begin{equation}\label{Frela}
F_2=F_1+\partial_t^3\psi+\alpha_1\partial_t^2\psi-\Delta_g\partial_t\psi-c^2\Delta_g\psi
+q_1\partial_t\psi+\partial_t(\beta(\partial_t \psi)^2).
\end{equation}
This completes the proof.
\end{proof}

Next, we provide the proof of Corollary \ref{cor1.3}.
\begin{proof}
\textbf{Case i}: $\alpha_1=\alpha_2$ in $\mathcal Q$.  $\beta\neq 0$ in $\mathcal Q$. Using \eqref{rela}, we obtain
\begin{equation*}
2\beta\partial_t \psi=0.
\end{equation*}
Since $\beta\neq0$, we have $\partial_t \psi=0$, and hence 
\begin{equation*}
\psi=\psi(x).
\end{equation*}
Since $\psi(0,x)=0$, it follows that $\psi\equiv0$ in $\mathcal Q$. Thus we obtain the unique recovery of $(\alpha,q,\beta,F)$.

\textbf{Case ii}:  $q_1=q_2$, $\beta\neq 0$ in $\mathcal Q$.   Using \eqref{rela}, we obtain
\begin{equation*}
2\partial_t(\beta \partial_t \psi)=0.
\end{equation*}
Then we have 
\begin{equation*}
\beta(t,x) \partial_t \psi(t,x)=\phi(x).
\end{equation*}
Evaluating at $t=0$, we find $\phi(x)=\beta(0,x)\partial_t \psi(0,x)=0$. Using $\beta\neq0$ in $\mathcal Q$, we obtain $\partial_t\psi=0$ in $\mathcal Q$. Hence
\begin{equation*}
\psi=\psi(x).
\end{equation*}
Using the argument in Case i, we finally get the unique recovery of $(\alpha,q,\beta,F)$.

\textbf{Case iii}: $F_1=F_2$ in $\mathcal Q$.
Using \eqref{Frela}, we have 
\begin{equation}\label{eq:psi-zero}
\partial_t^3\psi+\alpha_1 \partial_t^2\psi-\Delta_g\partial_t\psi-c^2\Delta_g\psi+q_1\partial_t\psi
+\partial_t\bigl(\beta (\partial_t\psi)^2\bigr)=0
\quad \text{in } \mathcal Q,
\end{equation}
with  $\psi|_{\Gamma}=(\partial_{\nu} \partial_t \psi+c^2 \partial_{\nu} \psi)|_{\Gamma}=0$, 
$\psi(0,\cdot)=\partial_t \psi(0,\cdot)=\partial_t^2 \psi(0,\cdot)=0$ in $M$.

It remains to show that
\[
\psi\equiv 0 \qquad \text{in } \mathcal Q.
\]
Since $\psi\in E^{m+2}$, we have $\partial_t\psi\in C([0,T];H^{m+1}(M))$. 
As $m>n+1$, the Sobolev embedding theorem implies
$H^{m+1}(M)\hookrightarrow C(M)$,
and hence 
\begin{equation*}
\partial_t\psi\in C([0,T]\times M).
\end{equation*}
Since $\psi=0$ on $\Gamma$, differentiating with respect to $t$ yields
$\partial_t\psi=\partial_t^2\psi=0  \text{ on } \Gamma$.
Multiplying \eqref{eq:psi-zero} by $\partial_t^2\psi$, integrating over $M$, and using integration by parts together with the boundary conditions, we obtain
\begin{align}
&\frac12\frac{d}{dt}\|\partial_t^2\psi(t)\|_{L^2(M)}^2
+\frac12\frac{d}{dt}\|\nabla_g\partial_t\psi(t)\|_{L^2(M)}^2
+\int_M \alpha_1 |\partial_t^2\psi|^2\,dV_g
+I_c
+\int_M q_1\,\partial_t\psi\,\partial_t^2\psi\,dV_g \nonumber \\
=&
-\int_M \partial_t\!\bigl(\beta (\partial_t\psi)^2\bigr)\,\partial_t^2\psi\,dV_g,
\label{eq:basic-energy}
\end{align}
where
\[
I_c:=-\int_M c^2\Delta_g\psi\,\partial_t^2\psi\,dV_g .
\]

We now rewrite $I_c$ exactly. Integrating by parts, we have
\[
I_c
=
\int_M c^2 \langle \nabla_g\psi,\nabla_g\partial_t^2\psi\rangle_g\,dV_g
+
\int_M \langle \nabla_g(c^2),\nabla_g\psi\rangle_g\,\partial_t^2\psi\,dV_g.
\]
Hence
\begin{align}
I_c
&=
\frac{d}{dt}\int_M c^2 \langle \nabla_g\psi,\nabla_g\partial_t\psi\rangle_g\,dV_g
-\int_M c^2 |\nabla_g\partial_t\psi|_g^2\,dV_g  
+\int_M \langle \nabla_g(c^2),\nabla_g\psi\rangle_g\,\partial_t^2\psi\,dV_g .
\label{eq:Ic-exact}
\end{align}

Substituting \eqref{eq:Ic-exact} into \eqref{eq:basic-energy}, we obtain
\begin{align}
&\frac{d}{dt}\Bigl[
\frac12\|\partial_t^2\psi(t)\|_{L^2(M)}^2
+\frac12\|\nabla_g\partial_t\psi(t)\|_{L^2(M)}^2
+\int_M c^2 \langle \nabla_g\psi,\nabla_g\partial_t\psi\rangle_g\,dV_g
\Bigr] \nonumber\\
=&
-\int_M \alpha_1 |\partial_t^2\psi|^2\,dV_g
+\int_M c^2 |\nabla_g\partial_t\psi|_g^2\,dV_g
 \nonumber\\
&
-\int_M \langle \nabla_g(c^2),\nabla_g\psi\rangle_g\,\partial_t^2\psi\,dV_g
-\int_M q_1\,\partial_t\psi\,\partial_t^2\psi\,dV_g
-\int_M \partial_t\!\bigl(\beta (\partial_t\psi)^2\bigr)\,\partial_t^2\psi\,dV_g .
\label{eq:energy-pre}
\end{align}

We next add a multiple of $\|\nabla_g\psi(t)\|_{L^2(M)}^2$. Define
\[
\mathcal E(t)
:=
\frac12\|\partial_t^2\psi(t)\|_{L^2(M)}^2
+\frac12\|\nabla_g\partial_t\psi(t)\|_{L^2(M)}^2
+\int_M c^2 \langle \nabla_g\psi,\nabla_g\partial_t\psi\rangle_g\,dV_g
+\frac{\mu}{2}\|\nabla_g\psi(t)\|_{L^2(M)}^2,
\]
where $\mu>0$ will be chosen later.

It follows from \eqref{eq:energy-pre} that
\begin{align}
\mathcal E'(t)
=&
-\int_M \alpha_1 |\partial_t^2\psi|^2\,dV_g
+\int_M c^2 |\nabla_g\partial_t\psi|_g^2\,dV_g
-\int_M \langle \nabla_g(c^2),\nabla_g\psi\rangle_g\,\partial_t^2\psi\,dV_g  \nonumber\\
&
-\int_M q_1\,\partial_t\psi\,\partial_t^2\psi\,dV_g -\int_M \partial_t\!\bigl(\beta (\partial_t\psi)^2\bigr)\,\partial_t^2\psi\,dV_g
+\mu\int_M \langle \nabla_g\psi,\nabla_g\partial_t\psi\rangle_g\,dV_g .
\label{eq:Eprime}
\end{align}

We now estimate the right-hand side term by term. By Cauchy--Schwarz and Young's inequality,
\begin{align}
\left|\int_M \langle \nabla_g(c^2),\nabla_g\psi\rangle_g\,\partial_t^2\psi\,dV_g\right|
&\le
C\Bigl(\|\nabla_g\psi\|_{L^2(M)}^2+\|\partial_t^2\psi\|_{L^2(M)}^2\Bigr), \label{eq:c2-est}\\
\mu\left|\int_M \langle \nabla_g\psi,\nabla_g\partial_t\psi\rangle_g\,dV_g\right|
&\le
C\Bigl(\|\nabla_g\psi\|_{L^2(M)}^2+\|\nabla_g\partial_t\psi\|_{L^2(M)}^2\Bigr). \label{eq:mu-est}
\end{align}

Since $\partial_t\psi=0$ on $\partial M$, Poincar\'e's inequality gives
\[
\|\partial_t\psi\|_{L^2(M)}
\le C\|\nabla_g\partial_t\psi\|_{L^2(M)}.
\]
Therefore,
\begin{equation}\label{eq:q-est-new}
\left|\int_M q_1 \partial_t\psi\,\partial_t^2\psi\,dV_g\right|
\le
C\Bigl(\|\nabla_g\partial_t\psi\|_{L^2(M)}^2+\|\partial_t^2\psi\|_{L^2(M)}^2\Bigr).
\end{equation}

For the nonlinear term, since
\[
\partial_t\bigl(\beta(\partial_t\psi)^2\bigr)
=
(\partial_t\beta)(\partial_t\psi)^2+2\beta\,\partial_t\psi\,\partial_t^2\psi,
\]
we have
\begin{align}
\left|\int_M \partial_t\bigl(\beta(\partial_t\psi)^2\bigr)\,\partial_t^2\psi\,dV_g\right|
&\le
C\|\partial_t\psi\|_{L^\infty(M)}
\Bigl(\|\partial_t\psi\|_{L^2(M)}\|\partial_t^2\psi\|_{L^2(M)}
+\|\partial_t^2\psi\|_{L^2(M)}^2\Bigr) \nonumber\\
&\le
C
\Bigl(\|\nabla_g\partial_t\psi\|_{L^2(M)}^2+\|\partial_t^2\psi\|_{L^2(M)}^2\Bigr).
\label{eq:nonlinear-est-new}
\end{align}

Combining \eqref{eq:Eprime}--\eqref{eq:nonlinear-est-new}, we obtain
\begin{align}
\mathcal E'(t)
\le
C
\Bigl(
\|\partial_t^2\psi(t)\|_{L^2(M)}^2
+\|\nabla_g\partial_t\psi(t)\|_{L^2(M)}^2
+\|\nabla_g\psi(t)\|_{L^2(M)}^2
\Bigr).
\label{eq:Eprime-est}
\end{align}

On the other hand, by Young's inequality and the boundedness of $c$,
\[
\left|
\int_M c^2 \langle \nabla_g\psi,\nabla_g\partial_t\psi\rangle_g\,dV_g
\right|
\le
\frac14\|\nabla_g\partial_t\psi(t)\|_{L^2(M)}^2
+C\|\nabla_g\psi(t)\|_{L^2(M)}^2.
\]
Hence, choosing $\mu>0$ sufficiently large, we obtain the coercivity estimate
\[
\mathcal E(t)\simeq
\|\partial_t^2\psi(t)\|_{L^2(M)}^2
+\|\nabla_g\partial_t\psi(t)\|_{L^2(M)}^2
+\|\nabla_g\psi(t)\|_{L^2(M)}^2 .
\]
Therefore, \eqref{eq:Eprime-est} yields
\[
\mathcal E'(t)
\le
C \mathcal E(t).
\]
Moreover,
\[
\mathcal E(0)=0.
\]
Hence Gr\"onwall's inequality implies
\[
\mathcal E(t)=0,\qquad t\in[0,T].
\]
By the coercivity of $\mathcal E$, it follows that
\[
\partial_t^2\psi=0,\qquad \nabla_g\partial_t\psi=0,\qquad \nabla_g\psi=0
\quad \text{in } \mathcal Q.
\]
Since $\nabla_g\psi=0$ in $M$ for each $t$ and $\psi|_{\partial M}=0$, it follows that $\psi\equiv0$ in $\mathcal Q$.
This proves the unique recovery of $(\alpha,q,\beta,F)$.

\end{proof}

\section*{Acknowledgements}
 X. Xu is partially supported by the National Key Research and Development Program of China (No. 2024YFA1012303), National Natural Science Foundation of China (No. 12525112), and the Open Research Project of Innovation Center of Yangtze River Delta, Zhejiang University. T. Zhou is partially supported by the National Key Research and Development Program of China (No. 2024YFA1012301), the Zhejiang Provincial Basic Public Welfare Research Program [
Grant Number LDQ24A010001], and NSFC Grant 12371426. The authors would like to thank Song-Ren Fu for his helpful discussions and comments.

\section*{Conflict Of Interest Statement}

The authors declare that they have no conflict of interest.

\section*{Data Availability Statement}

Data sharing is not applicable to this article as no datasets were generated
or analyzed during the current study.

\bibliographystyle{abbrv}
\bibliography{ref}
\end{document}